%% file: best-nbc.tex
\documentclass[12pt]{article}
\usepackage{hyperref}
\usepackage{color,array}
\usepackage{subcaption}
\makeatletter
\def\zapcolorreset{\let\reset@color\relax\ignorespaces}
\def\colorrows#1{{\aftergroup\zapcolorreset#1}\ignorespaces}
\makeatother

\newcolumntype{L}[1]{>{\raggedright\arraybackslash}p{#1}}
\newcolumntype{C}[1]{>{\centering\arraybackslash}p{#1}}
\newcolumntype{R}[1]{>{\raggedleft\arraybackslash}p{#1}}

\usepackage{amsmath,amssymb,amsfonts,amsthm,latexsym}
\usepackage{tikz}
\usetikzlibrary{calc}
\usetikzlibrary{intersections}
\usetikzlibrary{decorations.pathreplacing}
\usepackage{pgfplots}\pgfplotsset{compat=1.10}
\usepackage{mathtools}
\usepackage{enumitem}
\usepackage{bm}
\usepackage{rotating}
\usepackage{subcaption}
\usepackage[sort,nocompress]{cite}

\usepackage{epstopdf}
\newtheorem{theorem}{Theorem}
\newtheorem{lemma}{Lemma}

\newtheorem{conjecture}{Conjecture}

\def\green{g}

\begin{document}

\title{On best constants in $L^2$ approximation}
\author{Andrea Bressan\footnote{
Dipartimento di matematica,
Università di Pavia, via Ferrata 5, 27100 Pavia, Italia,
{\it email: andrea.bressan@unipv.it}}
\and
Michael S. Floater\footnote{
Department of Mathematics,
University of Oslo, Moltke Moes vei 35, 0851 Oslo, Norway,
{\it email: michaelf@math.uio.no}}
\and
Espen Sande\footnote{
Department of Mathematics, University of Rome Tor Vergata,
Via della Ricerca Scientifica 1, 00133 Rome, Italy
{\it email: sande@mat.uniroma2.it}.}
}
\maketitle

\begin{abstract}
In this paper we provide explicit upper and lower bounds on certain $L^2$ $n$-widths, i.e., best constants in $L^2$ approximation. We further describe a numerical method to compute these $n$-widths approximately, and prove that this method is superconvergent. Based on our numerical results we formulate a conjecture on the asymptotic behaviour of the $n$-widths. Finally we describe how the numerical method can be used to compute the breakpoints of the optimal spline spaces of Melkman and Micchelli, which have recently received renewed attention in the field of Isogeometric Analysis.
\end{abstract}

\smallskip

\noindent {\em Math Subject Classification: }
Primary: 41A15, 41A44, 
Secondary: 34B27  

\smallskip


\section{Introduction}

In this paper we consider the following $n$-width problem: determine the smallest constant $d_n$ for which there exists an $n$-dimensional subspace $X_n$ of $L^2(a,b)$ such that for all $u\in H^r(a,b)$, $r\ge 1$,
\begin{equation}\label{eq:n-width}
\min_{v\in X_n}\Vert u -v \Vert \le d_n \Vert u^{(r)}\Vert.
\end{equation}
This problem was originally studied in the first half of the previous century by Kolmogorov \cite{Kolmogorov:36}, who showed that  $d_n$, for $n\geq r$, corresponds to the $(n+1)$-st eigenvalue of a certain differential operator. In fact, for $r=1$ he showed that $d_n=(b-a)/(n\pi)$. 

The related problem of finding a subspace $X_n$ that achieves the smallest constant $d_n$ in \eqref{eq:n-width}, a so-called optimal subspace for $H^r(a,b)$, was also studied by Kolmogorov. He showed that an optimal subspace for $H^r(a,b)$ is the span of the first $n$ eigenfunctions of the mentioned eigenvalue problem. Further optimal subspaces were later found by Melkman and Micchelli \cite{Melkman:78} and by Floater and Sande \cite{Floater:17}. The $n$-width problem was a major topic in the 70s and 80s \cite{Ioffe:68,Karlin:68,Tikhomirov:90,Karlovitz:76,Micchelli:77,Melkman:78, Korneichuk:1991,Shadrin:90,Schultz:70} and there has been recent activity in the context of Isogeometric Analysis (IGA) \cite{Evans:2009,Floater:17,Floater:18,Floater:per,Sande:2019,Chan:2018}.

In \cite{Kolmogorov:36} Kolmogorov further claimed that for $(a,b)=(0,1)$, 
\begin{equation}\label{eq:Kolmogorov}
d_n = \Big(\frac{1}{\pi n}\Big)^r + O\Big(\frac{1}{n}\Big)^{r+1},
\end{equation}
as $n\to\infty$. In \cite[Chapter VII]{Pinkus:85}  Pinkus gives a proof that $d_n=O(n^{-r})$ as $n\to\infty$.
As far as we know, no further asymptotic results concerning $d_n$ are known.

Our first result is to provide the following bounds for $d_n$.
\begin{theorem}\label{thm:first-bound}
	For all $n\ge r$ we have
$$
  \frac{(n-r+1)\pi}{b-a}  \le d_n^{-1/r} \le \frac{n\pi}{b-a}.
$$
\end{theorem}
Note that in the case $(a,b)=(0,1)$ this theorem gives a proof of \eqref{eq:Kolmogorov}.

Isogeometric Galerkin methods have been previously used to approximate $n$-widths   \cite{Evans:2009}. Our tests suggest that these methods are adequate for small values of $r$ and $n$, e.g. up to $r,n\approx 5$.
The second contribution in this paper is to propose a simple superconvergent collocation method that computes a good approximation of $d_n$ for $r$ and $n$ in the range of tens. 
A byproduct of the numerical method is the approximate computation of the internal knots of the first optimal spline space of Melkman and Micchelli \cite{Melkman:78}. These knots were recently used by Chan and Evans \cite{Chan:2018} in their numerical method for wave propagation based on IGA. By using the optimal knots of Melkman and Micchelli they were able to improve the approximation properties and the maximum stable timestep for their method. Numerical methods to approximate these knots are therefore highly desirable in the context of IGA.

The results of our numerical method lead us to the following conjecture:
\begin{conjecture}\label{conjecture} 
For all $r\geq 1$, 
$$
\left |d_n^{-1/r}-\frac{(n-(r-1)/2)\pi}{b-a} \right |\to 0\quad \text{as}\quad n\to\infty.
$$
\end{conjecture}
In other words, it appears that $d_n^{-1/r}$ approaches the midpoint of the upper and lower bounds in Theorem \ref{thm:first-bound}.

\section{Proof of Theorem~\ref{thm:first-bound}}

First we prove the lower bound of $d_n$.
Let 
$$H^r_1=\{u\in H^r(a,b): u^{(k)}(a)=u^{(k)}(b)=0,\, 1\le k<r,\, k\text{ odd}\}.$$
As proved in \cite[Theorem~1]{Floater:18}, the $n$-width of $H^r_1$ is $(b-a)^r/(\pi n)^{r }$, i.e.,
for each $n\ge 1$ there exists $X_n$ such that 
\begin{equation} 
\min_{v\in X_n}\Vert u -v \Vert \le  \Big( \frac {b-a}{\pi n}\Big)^r \Vert u^{(r)}\Vert ,\qquad \forall u\in H^r_1,
\end{equation}
and for no $X_n$ the above inequality holds for  with a smaller constant.
Since $H^r_1$ is a subset of $H^r(a,b)$ we have
$$ 
d_n\ge\Big( \frac {b-a}{\pi n}\Big)^r.
$$

The upper bound for $d_n$ is obtained by making a specific choice of $X_n$.
Let $S$ be the space of $C^{r-2}$ piecewise polynomials of degree $r-1$ on the uniform partition of $(a,b)$ in $m$ segments. Then for all $u\in H^r(a,b)$ we have by \cite[Theorem 1]{Bressan:2019} (or \cite[Theorem 1]{Sande:2019}) that
\begin{equation} 
\min_{v\in S}\Vert u -v \Vert \le  \Big( \frac {b-a}{\pi m}\Big)^r \Vert u^{(r)}\Vert.
\end{equation}
Since $\dim S=m+r-1$, letting $n=\dim S\geq r$ we deduce that
$$
d_n\le\Big( \frac {b-a}{\pi (n-r+1)}\Big)^r.
$$
This completes the proof of Theorem~\ref{thm:first-bound}.

\section{Associated eigenvalue problems}

It follows from \cite{Kolmogorov:36,Melkman:78,Pinkus:85,Floater:17} that $d_n^{-2}$ is equal to the $(n+1)$-th smallest eigenvalue (counting the $r$ multiplicities of zero) of the eigenvalue problem
\begin{equation}\label{eq:other-eigenvalue-problem}
\left\{
\begin{aligned}
&(-1)^r\psi^{(2r)}(x) = \mu \psi(x) &&& x\in(a,b),\\
&\psi^{(k)}(a)=\psi^{(k)}(b)=0 &&& k=r,r+1,\dots,2r-1.
\end{aligned}
\right .
\end{equation}
Moreover, the eigenfunctions corresponding to the $n$ smallest eigenvalues span an optimal space $X_n$.
Note in particular, that $0$ is an eigenvalue of \eqref{eq:other-eigenvalue-problem} with multiplicity $r$ and the null space is the space of polynomials of degree $r-1$. 
It is straightforward to see that in the case $r=1$ the eigenvalue problem \eqref{eq:other-eigenvalue-problem} can be solved analytically, and in this case the eigenvalues are  
$$0,\left(\frac{\pi}{b-a}\right)^2, \left(\frac{2\pi}{b-a}\right)^2,\left(\frac{3\pi}{b-a}\right)^2\ldots,$$
and the eigenfunctions are 
$$1,\cos\left(\pi \frac{x-a}{b-a}\right), \cos\left(2\pi \frac{x-a}{b-a}\right), \cos\left(3\pi \frac{x-a}{b-a}\right), \ldots.$$

It can further be shown that for any $r$ the eigenvalue problem in \eqref{eq:other-eigenvalue-problem} has the same non-zero eigenvalues as the problem with Dirichlet boundary conditions  \cite{Floater:17}:
\begin{equation}\label{eq:eigenvalue-problem}
\left\{
\begin{aligned}
&(-1)^r\phi^{(2r)}(x) = \mu \phi(x) &&& x\in(a,b),\\
&\phi^{(k)}(a)=\phi^{(k)}(b)=0 &&& k=0,1,\dots,r-1.
\end{aligned}
\right .
\end{equation}
In particular, \eqref{eq:eigenvalue-problem} has a trivial null space. Thus, for $n\geq r$, $d_n^{-2}$ is the $(n+1-r)$-th eigenvalue of \eqref{eq:eigenvalue-problem}.
For $r=1$ the eigenvalues of \eqref{eq:eigenvalue-problem} are  
$$\left(\frac{\pi}{b-a}\right)^2, \left(\frac{2\pi}{b-a}\right)^2,\left(\frac{3\pi}{b-a}\right)^2\ldots,$$
and the eigenfunctions are 
$$\sin\left(\pi \frac{x-a}{b-a}\right), \sin\left(2\pi \frac{x-a}{b-a}\right), \sin\left(3\pi \frac{x-a}{b-a}\right), \ldots.$$

Let $\green$ be the Green's function associated to \eqref{eq:eigenvalue-problem}.
Since the differential operator in \eqref{eq:eigenvalue-problem} is self-adjoint it follows that $\green$ is symmetric, i.e., $\green(x,y)=\green(y,x)$.
Moreover, it is the distributional solution of
\begin{equation}\label{eq:green}
\left\{
\begin{aligned}
  &\partial^{2r}_x \green(x,y) = (-1)^r\delta_{y},&&& x \in (a,b),
\\&\partial^k_x \green(x,y)|_{x\in\{a,b\}}= 0,&&& k=0,1,\ldots, r-1,\ y\in(a,b)
\end{aligned}\right .
\end{equation}
where $\delta_y$ is the Dirac's delta at $y$, i.e., $\delta_y f= f(y)$. 
\begin{lemma}
The  Green's function associated to \eqref{eq:eigenvalue-problem} satisfies
\begin{equation}\label{eq:kernel}
 \green(x,y) = \frac{(y-a)^r (b-y)^r}{(2r-1)!(b-a)} B[\underbrace{a,\dots,a}_{r\text{ times}},y,\underbrace{b,\dots,b}_{r\text{ times}}](x),
\end{equation}
where $B[\xi_1,\dots,\xi_n]$ is the B-spline with knots $\xi_1,\dots,\xi_n$.
\end{lemma}
\begin{proof}
Equation \eqref{eq:green} implies that $\green(\cdot,y)$ is a piecewise polynomial of degree $2r-1$ with unit jump of the $(2r-1)$-th derivative. The boundary conditions in \eqref{eq:green} are satisfied by the $r$ repetitions of the knots $a$ and $b$. Using the jump formula for the highest order derivative of a B-spline in Lemma 3.22 of \cite{Lyche:notes}  we then obtain \eqref{eq:kernel},
since the factor in front of the B-spline normalizes the jump of the $(2r-1)$-th derivative of $g(\cdot,y)$.
\end{proof}

For example, for $r=1$,
$$ B[a,y,b](x) = \begin{cases} (x-a)/(y-a) & x \le y, \cr
                            (b-x)/(b-y) & x \ge y,
              \end{cases} $$
and so
$$ g(x,y) = \begin{cases} (x-a)(b-y)/(b-a) & x \le y, \cr
                            (y-a)(b-x)/(b-a) & x \ge y.
              \end{cases} $$
For $r=2$,
$$ B[a,a,y,b,b](x) = \frac{(x-a)^2}{(y-a)^2 (b-a)^2}
               \big( (b-a)(y-x) + 2(b-x) (y-a) \big),
  \quad x \le y, $$
and so
$$ g(x,y) = \frac{(x-a)^2 (b-y)^2}{6(b-a)^3}
               \big( (b-a)(y-x) + 2(b-x) (y-a) \big),
  \quad x \le y. $$
  By the symmetry of $g(x,y)$, we have
$$ g(x,y) =g(y,x), \quad y \le x. $$

For general $r$, we can evaluate $g(x,y)$
for $x \in [a,y]$ by first using the usual Cox-de Boor-Mansfield algorithm
to evaluate the B-spline in \eqref{eq:kernel} (after adding knots to both ends
of the knot vector appropriately) and
then multiplying by the scaling factor.
For $x \in [y,b]$, we just set $g(x,y) = g(y,x)$.

\section{A superconvergent numerical method}

\def\vv{\bm{v}_\phi}
\def\Green{{\mathtt G}}

In this section we describe a simple numerical method to approximate the eigenvalues and eigenfunctions of \eqref{eq:eigenvalue-problem}. We remind that for $r=1$ the eigenvalues and eigenfunctions are known analytically and so we only need to consider the case $r\geq 2$.
Recalling that $\green$ is the Green's function to \eqref{eq:eigenvalue-problem} we discretize the following equivalent formulation of \eqref{eq:eigenvalue-problem} 
\begin{equation}\label{eq:eigenvalue-problem-inverse}
 \int_{a}^b \green(x,y)\phi(y)\,dy = \lambda \phi(x),
\end{equation}
where $\lambda=1/\mu$.
We consider the uniform partition of $[a,b]$ in $m+1$ segments with nodes $\Xi=[\xi_0,\dots,\xi_{m+1}]$
$$
\xi_k= a + k h, \qquad h = (b-a)/(m+1).
$$
By approximating the integral in \eqref{eq:eigenvalue-problem-inverse} with the trapezoidal rule on each segment $[\xi_i,\xi_{i+1}]$ and remembering $\green(x,a)=\green(x,b)=0$ we obtain the approximate equation
$$
\lambda \phi(x) =\int_{a}^b \green(x,y)\phi(y)\,dy\approx h \sum_{\ell=1}^m \green(x,\xi_\ell) \phi(\xi_\ell).
$$
Discretizing $\phi$ with its values $\vv$ at $\xi_1,\dots,\xi_m$ we obtain the finite dimensional eigenproblem
\begin{equation}\label{eq:discrete-eigen-problem}
\Green\vv=\lambda \vv,\qquad \Green:=[\green(\xi_k,\xi_\ell)]_{k,\ell=1,\dots,m}.
\end{equation}
The $k$-th eigenpair of \eqref{eq:discrete-eigen-problem} are then an approximation to the $k$-th eigenpair of \eqref{eq:eigenvalue-problem} and \eqref{eq:eigenvalue-problem-inverse}. We remark that such approximations of the eigenvalues of integral operators have been studied before \cite{Anselone:74,Atkinson:75,Chatelin:81,Osborn:75}. Let $(\lambda_k,\phi_k)$ be the $k$-th eigenpair of \eqref{eq:eigenvalue-problem-inverse} and $(\lambda_{k,h},\phi_{k,h})$ be the $k$-th eigenpair of  \eqref{eq:discrete-eigen-problem} such that $\max_{x\in[a,b]}|\phi_k(x)|= \max_{x\in[a,b]}|\phi_{k,h}(x)|=1$. From \cite{Atkinson:75} and \cite{Osborn:75} we then obtain the following error estimate for our eigenvalues and eigenfunctions: 
 for all $r$ and $k$ there exists a constant $C>0$ such that for all $h<(k+1)^{-1}$
we have
\begin{equation}\label{ineq:eigen-error}
\begin{aligned}
|\lambda_k - \lambda_{k,h}| &\leq C \max_{x\in[a,b]} |E_k(x)|,
\\
\max_{x\in[a,b]}|\phi_k(x) - \phi_{k,h}(x)| &\leq C \max_{x\in[a,b]} |E_k(x)|,
\end{aligned}
\end{equation}
where $E_k(x)$ is the quadrature error of the trapezoidal rule
$$E_k(x):=\int_{a}^bf_k(x,y)\,dy - h \sum_{\ell=1}^m f_k(x,\xi_\ell),$$
and $f_k(x,y):=\green(x,y)\phi_k(y)$.
Since $f_k(x,\cdot)\in C^2([a,b])$ for all $r\geq2$ we expect the error to be $O(h^2)$ as $h\to0$.
However, for $r\geq3$, the method is superconvergent as we prove by the
following argument. 
The function $\green(x,\cdot)\in C^{2r-2}([a,b])$ for each $x\in[a,b]$ and since the eigenfunction $\phi_k\in C^{\infty}([a,b])$ we have $f_k(x,\cdot)\in C^{2r-2}([a,b])$. Now for any $p\leq r-2$ we can use the Euler-Maclaurin expansion \cite[Section 3.4.5]{Dahlquist:2008} to obtain
\begin{align*}
E_k(x)=&-\sum_{j=1}^{p} \frac{B_{2j}}{(2j)!}h^{2j} (\partial_y^{2j-1}f_k(x,b)- \partial_y^{2j-1}f_k(x,a))
\\
&- \frac{B_{2p+2}}{(2p+2)!}h^{2p+2}\partial_y^{2p+2}f_k(x,\eta_x),
\end{align*}
for some $\eta_x\in(a,b)$, and where $B_j$ is the $j$-th Bernoulli number.
From the fact that $\partial_y^{j}\green(x,y)|_{y=a,b}=\partial_y^{j}\phi_k(y)|_{y=a,b}=0$ for all $j=0,\ldots, r-1$, we deduce that 
$\partial_y^{j}f_k(x,y)|_{y=a,b}=0$ for all $j=0,\ldots, 2r-1$.
We now let $p=r-2$ and take the maximum over all the $\eta_x\in[a,b]$ to obtain 
\begin{equation}\label{ineq:superconv}
|\lambda_k - \lambda_{k,h}|+\max_{x\in[a,b]}|\phi_k(x) - \phi_{k,h}(x)|\leq Ch^{2r-2} \max_{x,y\in[a,b]} |\partial_y^{2r-2}f_k(x,y)|.
\end{equation}
Recalling that $d_n=\lambda_{n+1-r}^{1/2}$ this implies that for some constant
$C>0$, depending only on $r$ and $n$, we have
\[
    |d_n-\lambda_{n+1-r,h}^{1/2}|\le C h^{2r-2}.  
\]
Even faster convergence appears in our numerical tests, suggesting that
the ``real'' order is $2r$; see Figs.~\ref{fig:r2-conv}--\ref{fig:r4-conv}. We also observe that the method seems to achieve machine precision for relatively small values of $m$.

The matrix $\Green$ in \eqref{eq:discrete-eigen-problem} is totally positive and symmetric, consequently only its upper triangular part needs to be computed. 
 As explained in the previous section, the coefficients of $\Green$ can be efficiently computed using the usual B-spline recurrence relation.
The numerical method allow us to compute $d_n$ for different values of $r$. 
Using a mesh of size $h=1/2048$ on $[0,1]$ we computed the first $6$ eigenvalues of $\Green$ for $r=1,\dots,20$.
These correspond to $d_n^{2}$ for $n=r,\dots,r+5$.
In Table~\ref{table-conjecture2} we report the error between the computed and conjectured approximation of $d_n^{-1/r}$.
The numbers in the table suggest Conjecture~\ref{conjecture}.

Melkman and Micchelli \cite{Melkman:78} proved that for all $r$ there exists an optimal space $X_n$ containing $C^{r-2}$ piecewise polynomials of degree $r-1$. The zeros of the eigenfunctions in \eqref{eq:eigenvalue-problem} are exactly the internal knots of these spline spaces. Thus, the presented numerical method allows us to compute these knots, and as stated in the introduction, this can be very useful for Isogeometric Analysis \cite{Chan:2018}. See Figure \ref{fig:eigfunc} for various computed eigenfunctions in the case $m=500$.

{
\begin{table}\begin{center}
\scriptsize
\input{code/result/rel_err_mr6mn30m2048.txt}
\end{center}
\caption{The relative difference between the conjectured and computed value of $d_n^{-1/r}$ for $r=1,\dots,20$ and $n=r,\dots,r+5$ when $h=2^{-11}$.}\label{table-conjecture2}
\end{table}
} 

\begin{figure}
\begin{center}
\begin{tikzpicture}[scale=.75, transform shape]
	\begin{loglogaxis}[
	minor x tick num=0,
	minor y tick num=9,
	xmin=0.001, xmax=0.25,
	ymin=1e-14, ymax=1,
	y=.25cm,
	x=2cm,
	grid=major,
	legend style={
	at={(0.25,0.50)},
	anchor=south,
	xlabel={Mesh size},
	ylabel={Error}
	}
	]

    \input{code/convergence/r2mn7.txt}

    \addplot+ [black,dashed,mark=*] plot[domain=0.00125:0.125,mark=none] {x^(2)};
    \addplot+ [black,dashed,mark=*] plot[domain=0.00125:0.125,mark=none] {64*x^(4)};
	\legend{
	\texttt{n=2}\\
	\texttt{n=3}\\
	\texttt{n=4}\\
	\texttt{n=5}\\
	\texttt{n=6}\\
	\texttt{n=7}\\
	\texttt{n=8}\\
	}
	\end{loglogaxis}\end{tikzpicture}

	\end{center}
	\caption{Convergence of $(\lambda_{h,n+1-r})^{1/2}$ to $d_n$ for $r=2$ and $n=2,\dots,8$ in terms of $h$. The result with $h=2^{-11}$ was taken as a reference.
    The dashed lines are $h^{2}$ and $h^{4}$. Note that $d_8\approx 10^{-3}$ in this case.}\label{fig:r2-conv}
\end{figure}

\begin{figure}
\begin{center}
\begin{tikzpicture}[scale=.75, transform shape]
	\begin{loglogaxis}[
	minor x tick num=0,
	minor y tick num=9,
	xmin=0.001, xmax=0.25,
	ymin=1e-19, ymax=1e-3,
	y=.25cm,
	x=2cm,
	grid=major,
	legend style={
	at={(0.25,0.50)},
	anchor=south,
	xlabel={Mesh size},
	ylabel={Error}
	}
	] 
	
    \input{code/convergence/r3mn7.txt}

    \addplot+ [black,dashed,mark=*] plot[domain=0.00125:0.125,mark=none] {x^(4)};
    \addplot+ [black,dashed,mark=*] plot[domain=0.00125:0.125,mark=none] {64*x^(6)};
	\legend{
	\texttt{n=3}\\
	\texttt{n=4}\\
	\texttt{n=5}\\
	\texttt{n=6}\\
	\texttt{n=7}\\
	\texttt{n=8}\\
	\texttt{n=9}\\
	}
	\end{loglogaxis}\end{tikzpicture}

	\end{center}
	\caption{Convergence of $(\lambda_{h,n+1-r})^{1/2}$ to $d_n$ for $r=3$ and $n=3,\dots,9$ in terms of $h$. The result with $h=2^{-11}$ was taken as a reference.
    The dashed lines are $h^{4}$ and $h^{6}$. Note that $d_8\approx 10^{-4}$ in this case.}\label{fig:r3-conv}
\end{figure}

\begin{figure}
\begin{center}
\begin{tikzpicture}[scale=.75, transform shape]
	\begin{loglogaxis}[
	minor x tick num=0,
	minor y tick num=9,
	xmin=0.001, xmax=0.25,
	ymin=1e-21, ymax=1e-3,
	y=.25cm,
	x=2cm,
	grid=major,
	legend style={
	at={(0.25,0.50)},
	anchor=south,
	xlabel={Mesh size},
	ylabel={Error}
	}
	] 
	
    \input{code/convergence/r4mn7.txt}

    \addplot+ [black,dashed,mark=*] plot[domain=0.00125:0.125,mark=none] {x^(6)};
    \addplot+ [black,dashed,mark=*] plot[domain=0.00125:0.125,mark=none] {64*x^(8)};
	\legend{
	\texttt{n=4}\\
	\texttt{n=5}\\
	\texttt{n=6}\\
	\texttt{n=7}\\
	\texttt{n=8}\\
	\texttt{n=9}\\
	\texttt{n=10}\\
	}
	\end{loglogaxis}\end{tikzpicture}

	\end{center}
	\caption{Convergence of $(\lambda_{h,n+1-r})^{1/2}$ to $d_n$ for $r=4$ and $n=4,\dots,10$ in terms of $h$. The result with $h=2^{-11}$ was taken as a reference.
    The dashed lines are $h^{6}$ and $h^{8}$. Note that $d_8\approx 10^{-5}$ in this case.}\label{fig:r4-conv}
\end{figure}

\begin{figure}
\begin{subfigure}{\textwidth}
\centering
\includegraphics[width=0.24\textwidth]{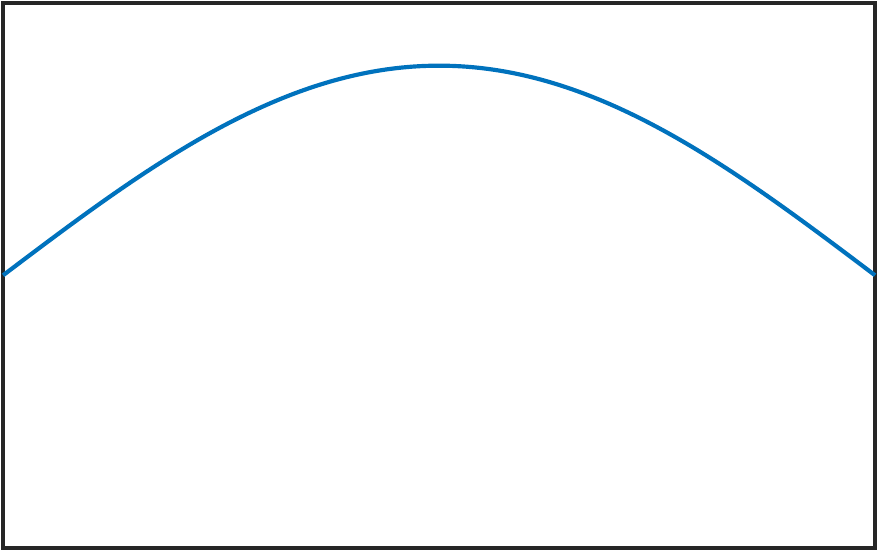}
\includegraphics[width=0.24\textwidth]{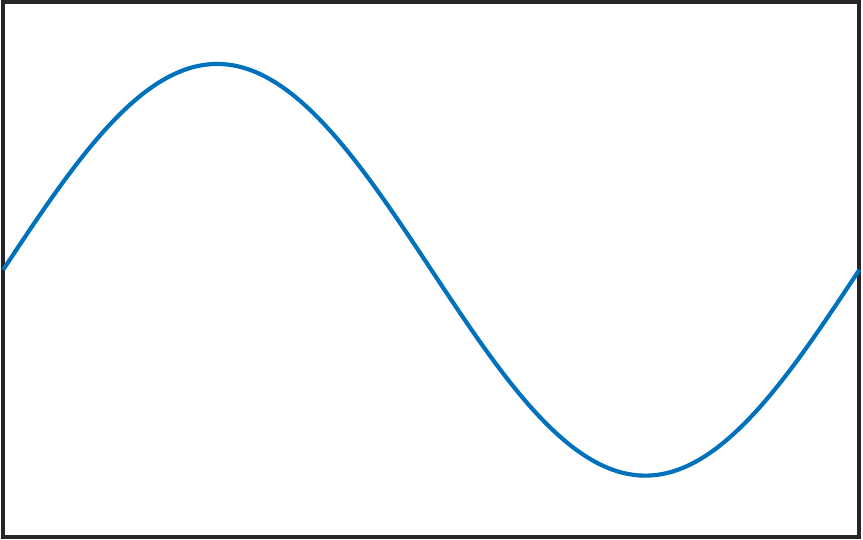}
\includegraphics[width=0.24\textwidth]{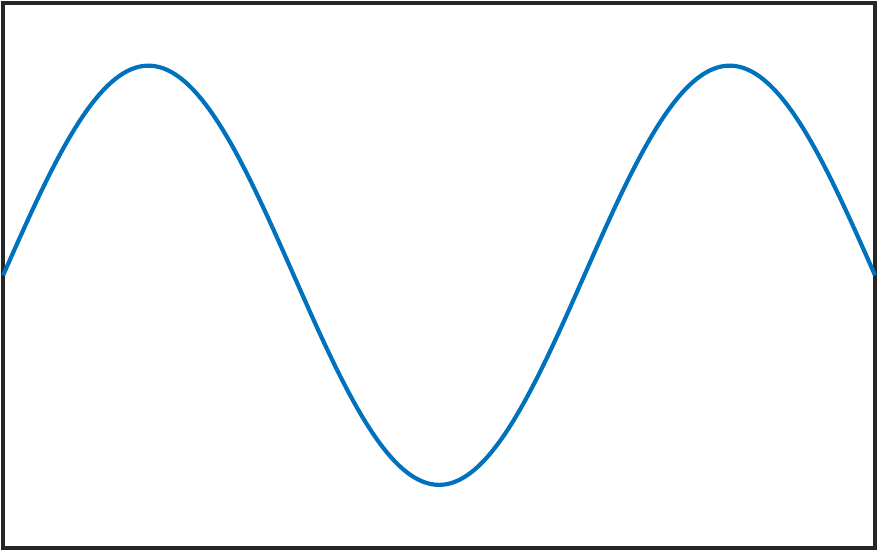}
\includegraphics[width=0.24\textwidth]{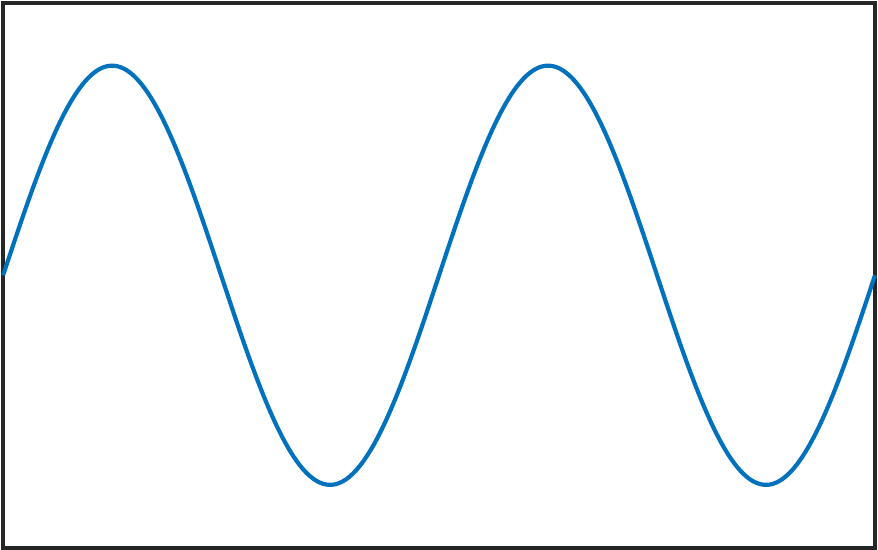}
\caption{First four eigenfunctions, $r=1$.}
\end{subfigure}
\begin{subfigure}{\textwidth}
\centering
\includegraphics[width=0.24\textwidth]{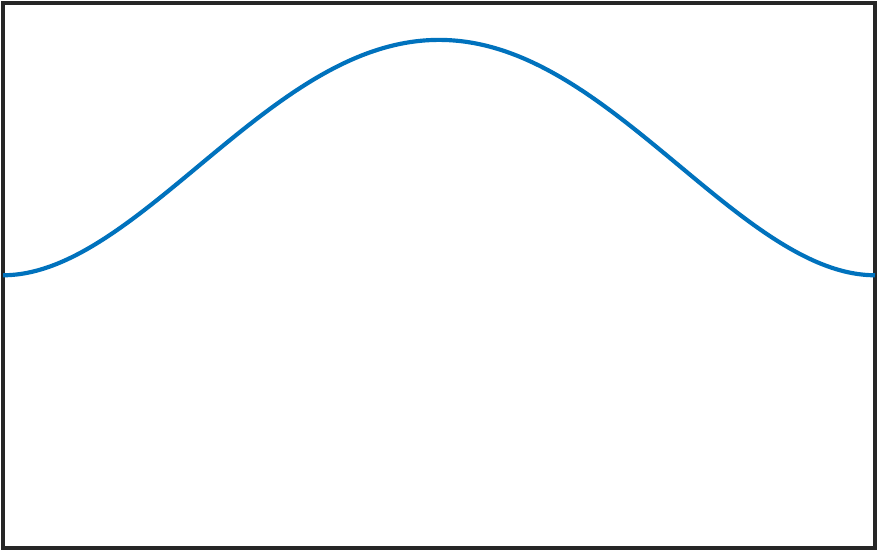}
\includegraphics[width=0.24\textwidth]{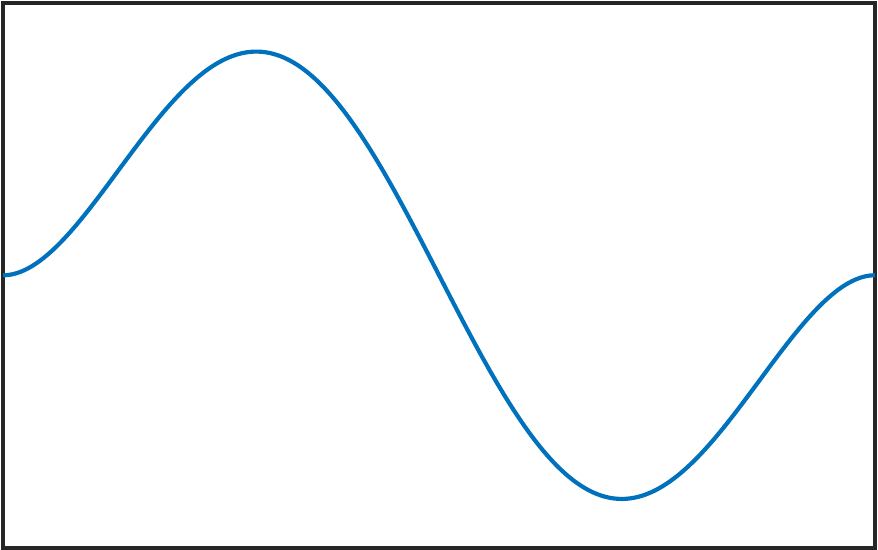}
\includegraphics[width=0.24\textwidth]{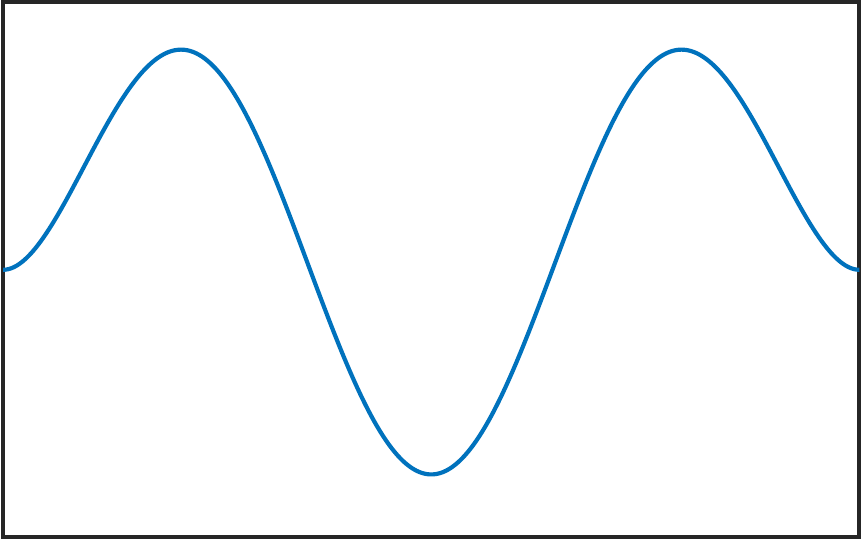}
\includegraphics[width=0.24\textwidth]{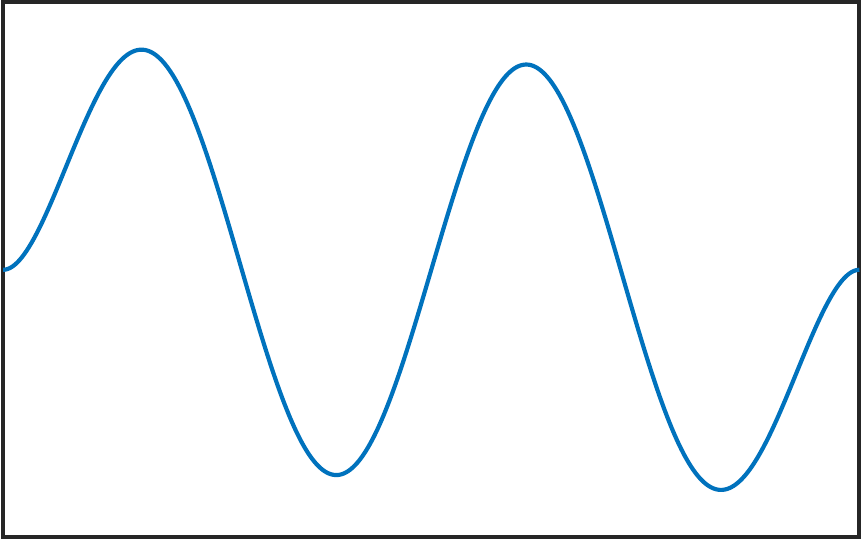}
\caption{First four eigenfunctions, $r=2$.}
\end{subfigure}
\begin{subfigure}{\textwidth}
\centering
\includegraphics[width=0.24\textwidth]{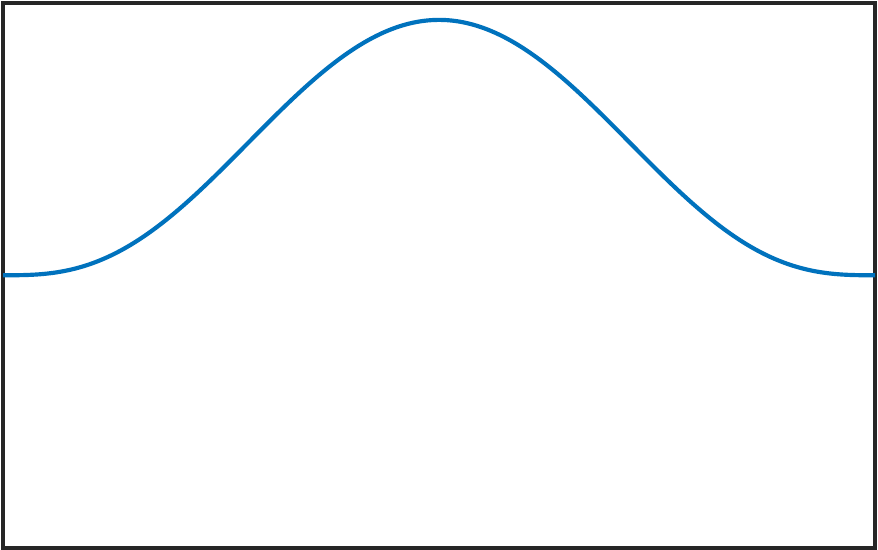}
\includegraphics[width=0.24\textwidth]{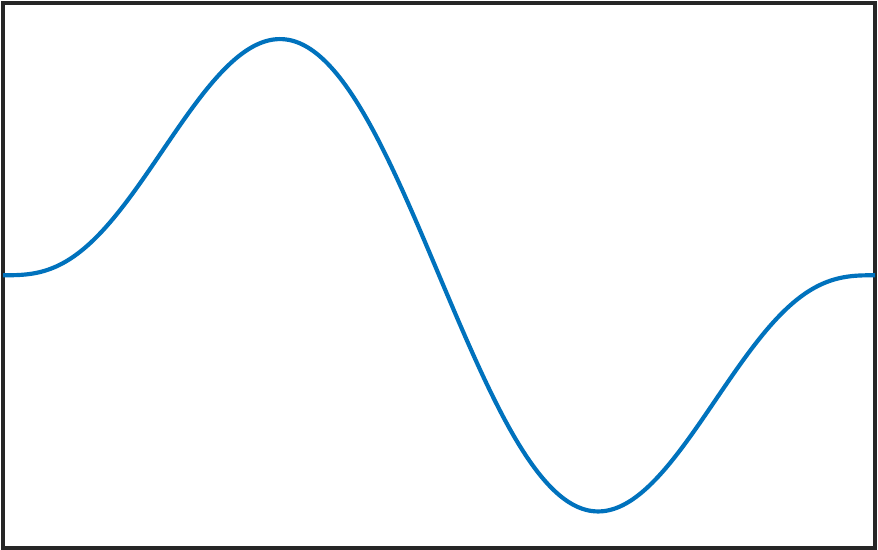}
\includegraphics[width=0.24\textwidth]{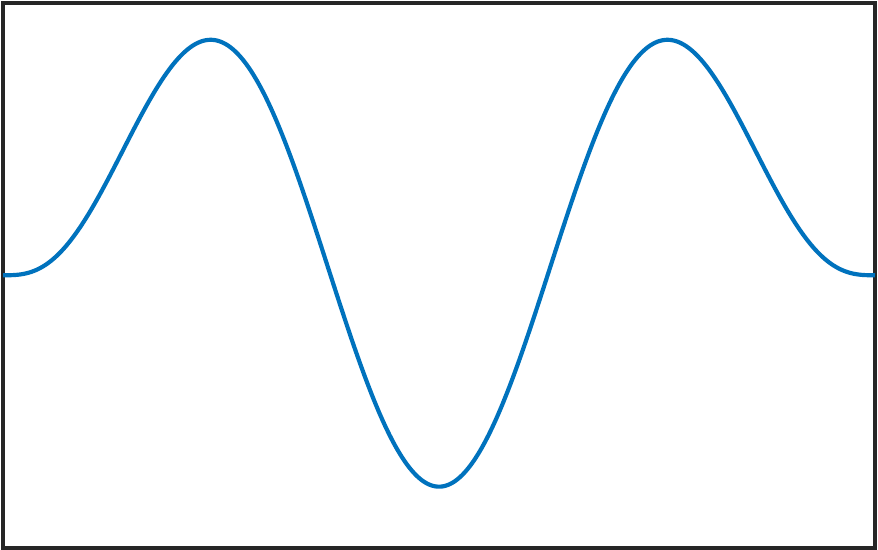}
\includegraphics[width=0.24\textwidth]{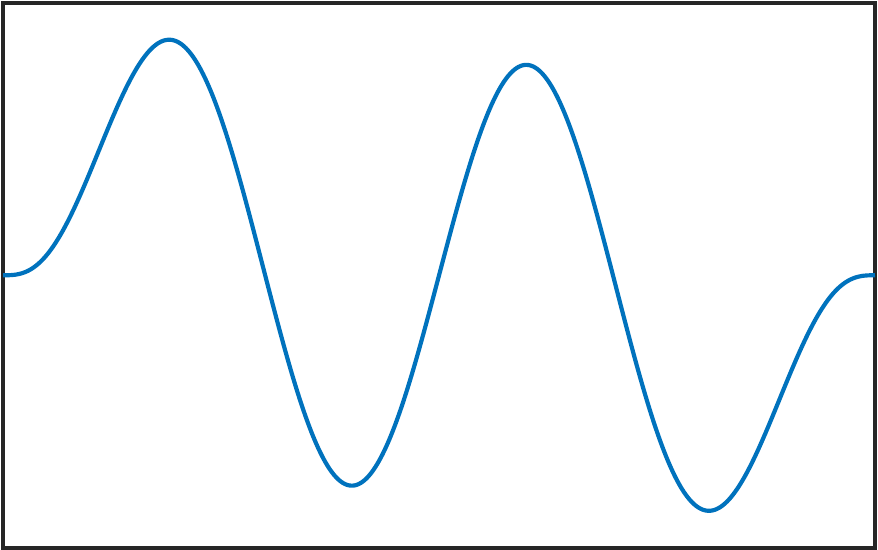}
\caption{First four eigenfunctions, $r=3$.}
\end{subfigure}
\begin{subfigure}{0.5\textwidth}
\centering
\includegraphics[width=0.48\textwidth]{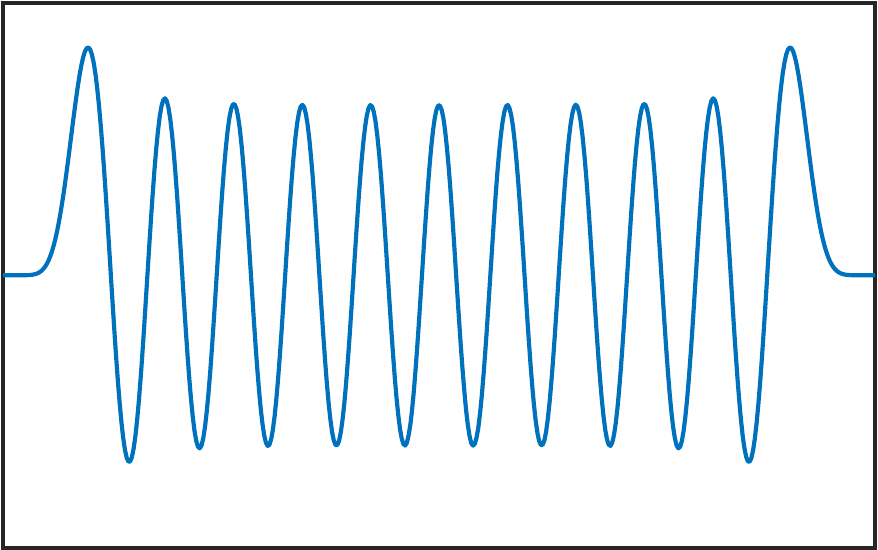}
\includegraphics[width=0.48\textwidth]{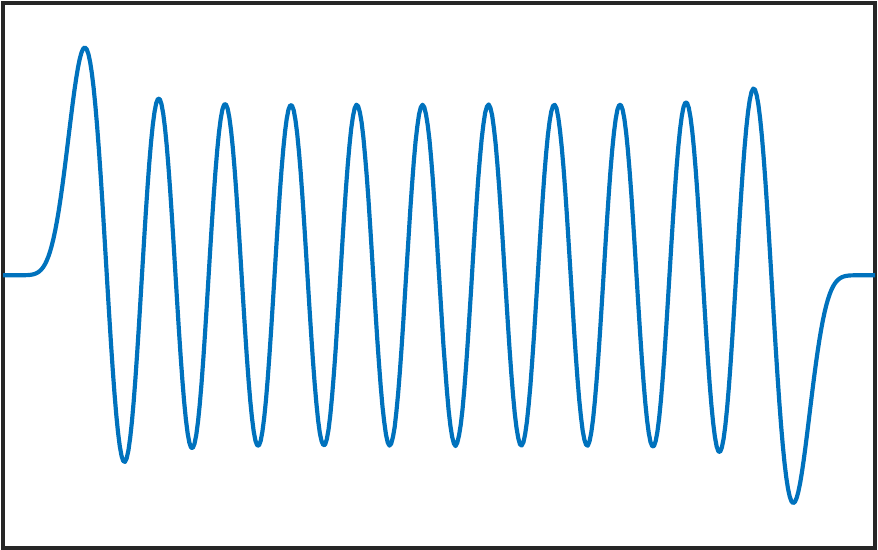}
\caption{Eigenfunctions $21$ and $22$, $r=10$.}
\end{subfigure}
\begin{subfigure}{0.5\textwidth}
\centering
\includegraphics[width=0.48\textwidth]{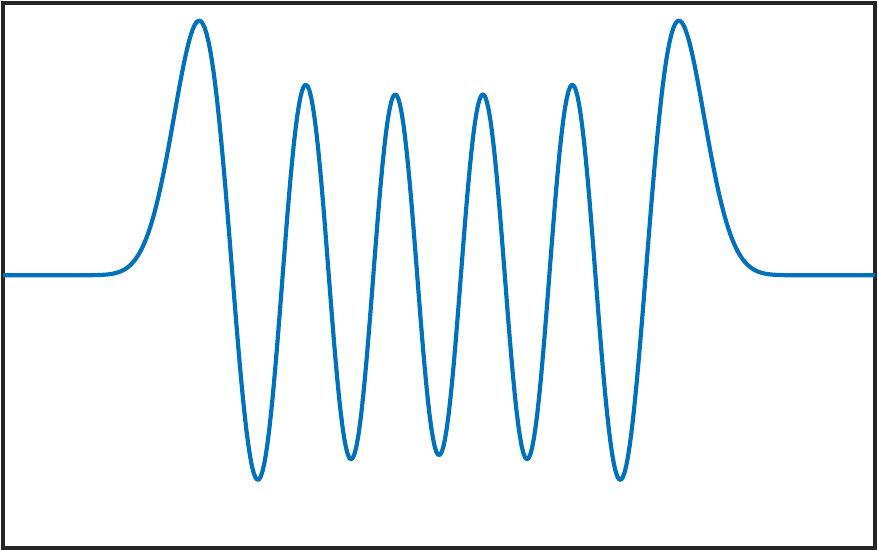}
\includegraphics[width=0.48\textwidth]{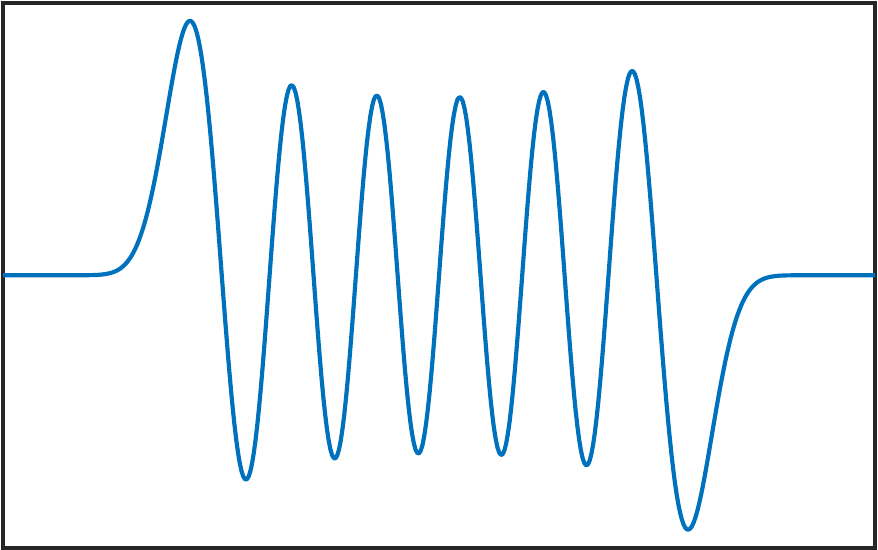}
\caption{Eigenfunctions $11$ and $12$, $r=20$.}
\end{subfigure}
\caption{Computed eigenfunctions on the interval $[a,b]=[-1,1]$ with $m=500$ for various choices of $r$ and $k$.}
\label{fig:eigfunc}
\end{figure}

\section*{Acknowledgements}
Andrea Bressan was partially supported by the European Research Council
through the FP7 Ideas Consolidator Grant \emph{HIGEOM} n.616563, and
by the Italian Ministry of Education, University and Research (MIUR)
through the  ``Dipartimenti di Eccellenza Program (2018-2022) - 
Dept. of Mathematics, University of Pavia''.
Espen Sande was supported 
by the Beyond Borders Programme of the University of Rome Tor Vergata through the project ASTRID (CUP E84I19002250005)
and by the MIUR Excellence Department Project awarded to the Department of Mathematics, University of Rome Tor Vergata (CUP E83C18000100006). Andrea Bressan and Espen Sande are members of Gruppo Nazionale per il Calcolo Scientifico, Istituto Nazionale di Alta Matematica.

\bibliographystyle{amsplain}
\bibliography{eigen}

\end{document}

%% file: code/result/rel_err_mr6mn30m2048.txt
\begin{tabular}{r|rrrrrr}
$n$\textbackslash$r$&$1$&$2$&$3$&$4$&$5$&$6$
\\\hline
\\ $r$& $-9.80e-08$& $3.75e-03$& $-3.33e-16$& $-4.49e-03$& $-8.65e-03$& $-1.23e-02$
\\ $r+1$& $-3.92e-07$& $-9.89e-05$& $2.42e-04$& $2.33e-05$& $-6.79e-04$& $-1.70e-03$
\\ $r+2$& $-8.82e-07$& $3.05e-06$& $-3.33e-16$& $3.76e-05$& $-1.95e-05$& $-2.23e-04$
\\ $r+3$& $-1.57e-06$& $-1.03e-07$& $-6.30e-07$& $3.65e-06$& $7.60e-06$& $-2.13e-05$
\\ $r+4$& $-2.45e-06$& $3.62e-09$& $-5.55e-16$& $1.07e-07$& $2.02e-06$& $1.42e-07$
\\ $r+5$& $-3.53e-06$& $-1.36e-10$& $1.95e-09$& $-1.95e-08$& $3.19e-07$& $7.51e-07$
\\ $r+6$& $-4.80e-06$& $-1.11e-12$& $-9.99e-16$& $-3.22e-09$& $3.47e-08$& $2.50e-07$
\\ $r+7$& $-6.27e-06$& $-1.02e-11$& $-6.57e-12$& $-1.98e-10$& $1.78e-09$& $5.95e-08$
\\ $r+8$& $-7.93e-06$& $-1.56e-11$& $-3.22e-15$& $7.20e-12$& $-2.85e-10$& $1.17e-08$
\\ $r+9$& $-9.79e-06$& $-2.33e-11$& $1.78e-14$& $2.83e-12$& $-1.05e-10$& $1.93e-09$
\\ $r+10$& $-1.19e-05$& $-3.36e-11$& $3.11e-15$& $2.35e-13$& $-1.93e-11$& $2.43e-10$
\\ $r+11$& $-1.41e-05$& $-4.69e-11$& $-9.77e-15$& $-8.20e-14$& $-1.97e-12$& $1.52e-11$
\\ $r+12$& $-1.66e-05$& $-6.37e-11$& $-1.24e-14$& $-5.33e-14$& $-1.36e-12$& $-1.21e-11$
\\ $r+13$& $-1.92e-05$& $-8.48e-11$& $-7.03e-14$& $2.22e-16$& $2.70e-13$& $-8.33e-12$
\\ $r+14$& $-2.20e-05$& $-1.11e-10$& $6.00e-15$& $-1.09e-12$& $-5.37e-13$& $-6.12e-12$
\\ $r+15$& $-2.51e-05$& $-1.42e-10$& $1.67e-14$& $1.48e-12$& $1.07e-11$& $-4.82e-11$
\\ $r+16$& $-2.83e-05$& $-1.80e-10$& $1.60e-14$& $6.26e-14$& $1.19e-11$& $-5.00e-11$
\\ $r+17$& $-3.17e-05$& $-2.25e-10$& $-1.25e-14$& $-1.55e-13$& $1.76e-11$& $8.50e-11$
\\ $r+18$& $-3.54e-05$& $-2.78e-10$& $-2.79e-14$& $2.03e-12$& $-3.13e-11$& $-3.25e-10$
\\ $r+19$& $-3.92e-05$& $-3.39e-10$& $-3.19e-13$& $3.85e-12$& $1.38e-11$& $-5.46e-10$
\\ $r+20$& $-4.32e-05$& $-4.10e-10$& $2.74e-13$& $-1.78e-13$& $-7.52e-12$& $-6.63e-10$
\\ $r+21$& $-4.74e-05$& $-4.92e-10$& $-4.04e-13$& $4.82e-13$& $-4.22e-11$& $3.85e-10$
\\ $r+22$& $-5.18e-05$& $-5.85e-10$& $3.02e-13$& $1.30e-11$& $2.68e-11$& $4.53e-11$
\\ $r+23$& $-5.64e-05$& $-6.92e-10$& $4.88e-14$& $-1.09e-11$& $7.96e-11$& $-3.42e-10$
\\ $r+24$& $-6.12e-05$& $-8.12e-10$& $-8.22e-13$& $1.75e-11$& $5.40e-10$& $5.61e-09$
\\ $r+25$& $-6.62e-05$& $-9.47e-10$& $2.58e-13$& $1.81e-11$& $5.76e-10$& $3.92e-09$
\\ $r+26$& $-7.14e-05$& $-1.10e-09$& $2.63e-13$& $-6.21e-12$& $-1.75e-10$& $1.52e-09$
\\ $r+27$& $-7.68e-05$& $-1.27e-09$& $-2.89e-13$& $-2.32e-11$& $-2.74e-10$& $-2.08e-09$
\\ $r+28$& $-8.24e-05$& $-1.45e-09$& $2.16e-13$& $-2.44e-11$& $-1.89e-09$& $1.64e-09$
\\ $r+29$& $-8.82e-05$& $-1.66e-09$& $7.63e-13$& $3.21e-11$& $5.23e-10$& $-2.31e-08$
\end{tabular}

%% file: code/convergence/r2mn7.txt
\addplot+ coordinates{
(0.125,2.97531e-06)
(0.0625,1.9933e-07)
(0.03125,1.33496e-08)
(0.015625,8.7464e-10)
(0.0078125,5.61745e-11)
(0.00390625,3.56176e-12)
(0.00195312,2.23425e-13)
};\addplot+ coordinates{
(0.125,9.9129e-06)
(0.0625,5.85864e-07)
(0.03125,3.74729e-08)
(0.015625,2.4225e-09)
(0.0078125,1.55036e-10)
(0.00390625,9.8211e-12)
(0.00195312,6.16035e-13)
};\addplot+ coordinates{
(0.125,2.38409e-05)
(0.0625,1.23672e-06)
(0.03125,7.50701e-08)
(0.015625,4.77644e-09)
(0.0078125,3.04379e-10)
(0.00390625,1.92604e-11)
(0.00195312,1.20783e-12)
};\addplot+ coordinates{
(0.125,4.92402e-05)
(0.0625,2.22411e-06)
(0.03125,1.27294e-07)
(0.015625,7.94979e-09)
(0.0078125,5.04037e-10)
(0.00390625,3.18526e-11)
(0.00195312,1.99683e-12)
};\addplot+ coordinates{
(0.125,9.33936e-05)
(0.0625,3.64846e-06)
(0.03125,1.95763e-07)
(0.015625,1.19697e-08)
(0.0078125,7.54476e-10)
(0.00390625,4.76067e-11)
(0.00195312,2.9833e-12)
};\addplot+ coordinates{
(0.125,0.000167143)
(0.0625,5.64489e-06)
(0.03125,2.82468e-07)
(0.015625,1.6868e-08)
(0.0078125,1.05621e-09)
(0.00390625,6.65309e-11)
(0.00195312,4.16737e-12)
};\addplot+ coordinates{
(0.125,0.000282998)
(0.0625,8.39424e-06)
(0.03125,3.89817e-07)
(0.015625,2.26825e-08)
(0.0078125,1.40984e-09)
(0.00390625,8.86348e-11)
(0.00195312,5.54918e-12)
};

%% file: code/convergence/r3mn7.txt
\addplot+ coordinates{
(0.125,1.20472e-08)
(0.0625,4.33763e-11)
(0.03125,2.54164e-12)
(0.015625,5.15742e-14)
(0.0078125,8.7777e-16)
(0.00390625,1.30104e-17)
(0.00195312,2.60209e-18)
};\addplot+ coordinates{
(0.125,6.19393e-08)
(0.0625,2.60737e-11)
(0.03125,7.67051e-12)
(0.015625,1.7009e-13)
(0.0078125,2.94881e-15)
(0.00390625,4.74881e-17)
(0.00195312,8.67362e-19)
};\addplot+ coordinates{
(0.125,1.62891e-07)
(0.0625,3.76358e-10)
(0.03125,1.63154e-11)
(0.015625,3.94362e-13)
(0.0078125,6.95049e-15)
(0.00390625,1.13191e-16)
(0.00195312,1.95156e-18)
};\addplot+ coordinates{
(0.125,2.53713e-07)
(0.0625,1.12717e-09)
(0.03125,2.88862e-11)
(0.015625,7.55895e-13)
(0.0078125,1.3514e-14)
(0.00390625,2.20852e-16)
(0.00195312,3.84892e-18)
};\addplot+ coordinates{
(0.125,7.90846e-08)
(0.0625,2.15168e-09)
(0.03125,4.59651e-11)
(0.015625,1.28543e-12)
(0.0078125,2.32623e-14)
(0.00390625,3.81368e-16)
(0.00195312,6.09864e-18)
};\addplot+ coordinates{
(0.125,1.06997e-06)
(0.0625,2.80465e-09)
(0.03125,6.88177e-11)
(0.015625,2.01488e-12)
(0.0078125,3.68218e-14)
(0.00390625,6.04023e-16)
(0.00195312,9.12085e-18)
};\addplot+ coordinates{
(0.125,4.77127e-06)
(0.0625,1.52449e-09)
(0.03125,9.99728e-11)
(0.015625,2.97852e-12)
(0.0078125,5.4826e-14)
(0.00390625,9.0211e-16)
(0.00195312,1.41353e-17)
};

%% file: code/convergence/r4mn7.txt
\addplot+ coordinates{
(0.125,7.36727e-10)
(0.0625,1.62793e-12)
(0.03125,3.02509e-15)
(0.015625,7.10152e-18)
(0.0078125,5.42101e-20)
(0.00390625,1.6263e-19)
(0.00195312,1.0842e-19)
};\addplot+ coordinates{
(0.125,3.30163e-09)
(0.0625,8.31892e-12)
(0.03125,1.48978e-14)
(0.015625,3.13199e-17)
(0.0078125,1.6263e-19)
(0.00390625,5.42101e-20)
(0.00195312,1.0842e-19)
};\addplot+ coordinates{
(0.125,8.50788e-09)
(0.0625,2.71364e-11)
(0.03125,4.88275e-14)
(0.015625,9.54166e-17)
(0.0078125,2.5411e-19)
(0.00390625,3.38813e-20)
(0.00195312,1.01644e-20)
};\addplot+ coordinates{
(0.125,1.44502e-08)
(0.0625,6.79701e-11)
(0.03125,1.26393e-13)
(0.015625,2.35228e-16)
(0.0078125,6.1664e-19)
(0.00390625,4.06576e-20)
(0.00195312,2.20229e-20)
};\addplot+ coordinates{
(0.125,1.62349e-08)
(0.0625,1.41949e-10)
(0.03125,2.79073e-13)
(0.015625,5.02651e-16)
(0.0078125,1.22142e-18)
(0.00390625,1.69407e-21)
(0.00195312,4.31987e-20)
};\addplot+ coordinates{
(0.125,2.03633e-08)
(0.0625,2.58973e-10)
(0.03125,5.49116e-13)
(0.015625,9.68918e-16)
(0.0078125,2.23828e-18)
(0.00390625,3.26108e-20)
(0.00195312,3.47284e-20)
};\addplot+ coordinates{
(0.125,2.59748e-07)
(0.0625,4.25602e-10)
(0.03125,9.89837e-13)
(0.015625,1.72737e-15)
(0.0078125,3.9099e-18)
(0.00390625,1.14349e-20)
(0.00195312,3.04932e-20)
};